\documentclass[10pt,reqno]{amsart}
\usepackage{amscd,amsmath,amssymb,amsthm}
\newtheorem{theorem}{Theorem}

\newtheorem{cor}[theorem]{Corollary}
\newtheorem*{cor*}{Corollary}
\newtheorem{lemma}[theorem]{Lemma}
\newtheorem{fact}[theorem]{Fact}

\newtheorem{proposition}[theorem]{Proposition}
\theoremstyle{definition}
\newtheorem{defi}[theorem]{Definition}
\theoremstyle{remark}
\newtheorem*{remark*}{Remark}
\newtheorem*{remarks*}{Remarks}
\newtheorem{example}[theorem]{Example}
\newcommand{\comment}[1]{}

\newcommand{\Z}{\mathbb Z}

\title[Reflections on equicontinuity]{Reflections on equicontinuity}

\author{Joseph Auslander}
\address{Department of Mathematics, University of Maryland, College Park, MD 20742, USA}
\email{jna@math.umd.edu}
\author{Gernot Greschonig}
\address{Institute of Discrete Mathematics and Geometry, Vienna University of Technology, Wiedner Hauptstra\ss e 8-10, A-1040 Vienna, Austria}
\email{greschg@fastmail.net}
\author{Anima Nagar}
\address{Department of Mathematics, Indian Institute of Technology, Hauz Khas, New Delhi-110016, India}
\email{anima@maths.iitd.ac.in }

\thanks{The second author was supported by the research project S9614 of the Austrian Science Fund (FWF)}
\begin{document}

\begin{abstract}
We study different conditions which turn out to be equivalent to equicontinuity for a transitive compact Hausdorff flow with a general group action.
Among them are a notion of ``regional'' equicontinuity, also known as ``Furstenberg'' condition, and the condition that every point of the phase space is almost automorphic.
Then we study relations on the phase space arising from dynamical properties, among them the regionally proximal relation and two relations introduced by Veech.
We generalize Veech's results for minimal actions of non-Abelian groups preserving a probability measure with respect to the regionally proximal relation.
We provide proofs in the framework of dynamical systems rather than harmonic analysis as given by Veech.
\end{abstract}
\maketitle

Throughout this note let $X$ be a compact Hausdorff space and $T$ a topological group continuously acting on the left of  $X$, so that $(X,T)$ is a compact Hausdorff flow.
There exists a unique uniform structure on $X$, defined by a set $\mathcal U$ of \emph{entourages} coincident with the set of neighborhoods of the diagonal $\Delta=\{(x,x):x\in X\}$ in the product topology on $X\times X$.
Given an entourage $\alpha\in\mathcal U$ and a point $x\in X$ we define the $\alpha$-neighborhood of $x$ by $\{y\in X:(y,x)\in\alpha\}$, and using these sets as a neighborhood base at $x$ the uniform structure defines the topology on $X$.
Given entourages $\alpha, \beta\in\mathcal U$ we define the product by
\begin{equation*}
\alpha\beta=\{(x,z)\in X\times X:\textup{ there exists }y\in X \textup{ s. th. } (x,y)\in\alpha\textup{ and } (y,z)\in\beta\}
\end{equation*}
and the inverse by $\alpha^{-1}=\{(y,x)\in X\times X:(x,y)\in\alpha\}$.
If $d$ is a metric on $X$ compatible with the topology, then a set $\alpha\subset X\times X$ is an element of $\mathcal U$ if and only if there exists an $\varepsilon>0$ so that $\{(x,y)\in X\times X:d(x,y)<\varepsilon\}\subset\alpha$.
For a general (non-metrizable) uniform structure, there is always a family $\mathcal D=\{d_i:i\in I\}$ of pseudometrics on $X$ so that $\alpha\subset X\times X$ is an element of $\mathcal U$ if and only if there exists an element $d_i\in\mathcal D$ and $\varepsilon>0$ with $\{(x,y)\in X\times X:d_i (x,y)<\varepsilon\}\subset\alpha$.
We call a point $x\in X$ an \emph{equicontinuity point}, if for every entourage $\alpha\in\mathcal U$ there exists a neighborhood $U_x$ of $x$ so that for every $y\in U_x$ and every $t\in T$ it holds that $(tx,ty)\in\alpha$.
If every point in $X$ is an equicontinuity point, then by a finite covering argument for every entourage $\alpha\in\mathcal U$ there exists a entourage $\beta\in\mathcal U$ so that $(x,y)\in\beta$ implies $(tx,ty)\in\alpha$ for all $t\in T$.
Such a flow $(X,T)$ is called \emph{equicontinuous}.
We denote the orbit closure of a point $x\in X$ by $\bar{\mathcal O}_T(x)$, hence a flow is transitive if $\bar{\mathcal O}_T(x)=X$ for some $x\in X$ and minimal if $\bar{\mathcal O}_T(x)=X$ for all $x\in X$.
The closure of $T$ in the product space $X^X$ has a natural semigroup operation continuous in the left argument, the \emph{enveloping semigroup} $\mathcal E(X,T)$ of the flow $(X,T)$.

We call two points $x,y\in X$ \emph{proximal} if there exists a net $\{t_i\}_{i\in I}\subset T$ and a point $z\in X$ with $t_i x\to z$ and $t_i y\to z$, otherwise the points are called \emph{distal}.
A point in $X$ is called distal if it is distal to every distinct point, and a flow is called distal if any two distinct points are distal (i.e. every point in $X$ is a distal point).
It is easy to see that every equicontinuous flow is distal.
For a distal flow the enveloping semigroup is a group algebraically, but in general not a topological group.
Moreover, every distal flow admits a decomposition of its phase space into minimal orbit closures.

The main topic of this note are conditions closely related to equicontinuity.
Most of these conditions have been defined for compact metric phase spaces initially, however, all of them can easily be translated into the setting of uniform spaces.
Therefore we decided to present the results for a compact Hausdorff phase space, as long as no further condition is mentioned.

\begin{defi}
The flow $(X,T)$ fulfills the \emph{Furstenberg condition}, if for every entourage $\alpha\in\mathcal U$ there exists an entourage $\beta\in\mathcal U$ so that whenever $(t x,x)\in\beta$ and $(y,x)\in\beta$ for some $t\in T$ we have $(t y,x)\in\alpha$.
\end{defi}

\begin{remark*}
Hence a compact metric flow $(X,T)$ fulfills the \emph{Furstenberg condition}, if for every $\varepsilon>0$ there exists a universal $\delta>0$ so that whenever $d(t x,x)<\delta$ and $d(y,x)<\delta$ for some $t\in T$ we have $d(t y,x)<\varepsilon$.
\end{remark*}

\begin{defi}[cf. \cite{Ve1}]
A point $x\in X$ is called \emph{almost automorphic}, if for every net $\{t_i\}_{i\in I}\subset T$ with $t_i x\to y$ for some $y\in X$ it holds that $t_i^{-1} y\to x$.
An almost automorphic cannot be proximal to a distinct point and is thus a distal point.
\end{defi}

For an Abelian group $T$, flows with almost automorphic points have been studied in the paper \cite{Ve1}.
The existence of an almost automorphic point is sufficient for a residual set of almost automorphic points, and then the flow is an almost automorphic (i.e. one to one onto a residual set) extension of its (always non-trivial) maximal equicontinuous factor.
Here in this note we use the stronger condition that \emph{every} point in $X$ is almost automorphic, however without the requirement that $T$ is Abelian.

\begin{theorem}\label{th:fc}
The Furstenberg condition is equivalent to equicontinuity of the flow, and these equivalent properties imply that every point in $X$ is an almost automorphic point.
\end{theorem}

\begin{proof}
Let $z\in X$ and the entourage $\alpha\in\mathcal U$ be arbitrary, choose an entourage $\alpha'\in\mathcal U$ with $\alpha'^2\subset\alpha$, and let the entourage $\beta\in\mathcal U$ with $\beta^{-1}\subset\alpha'$ be related to $\alpha'$ according to the Furstenberg condition.
The $\beta$-neighborhoods of the points in the $T$-orbit of $z$ define an open covering of the orbit closure $\bar{\mathcal O}_T(z)$, and hence there is a finite set $K\subset T$ so that the $\beta$-neighborhood of the set $\{tz:t\in K\}$ contains $\bar{\mathcal O}_T(z)$.
By the finiteness of $K$ we can choose an entourage $\beta'\in\mathcal U$ so that $(y,z)\in\beta'$ for $y\in X$ implies $(t y,t z)\in\beta$ for all $t\in K$.
Now let $\tau\in T$ be arbitrary, select an element $t'\in K$ with $(\tau z,t' z)\in\beta$, and put $x=t' z$ and $t=\tau t'^{-1}$.
Then
\begin{equation*}
(t x,x)=(\tau t'^{-1}t' z,t' z)=(\tau z,t' z)\in\beta ,
\end{equation*}
and for all $y\in X$ with $(y,z)\in\beta'$ it holds that $(t' y,x)\in\beta$.
By the Furstenberg condition, $(t x,x)\in\beta$, and $(t' y,x)\in\beta$, we have
\begin{equation*}
(\tau y,t'z)=(\tau t'^{-1}t' y,t'z)=(t(t' y),x)\in\alpha'.
\end{equation*}
That is
\begin{equation*}
(\tau y,\tau z)\in\alpha'\beta^{-1}\subset\alpha'^2\subset\alpha
\end{equation*}
for all $y\in X$ with $(y,z)\in\beta'$.
Since $\beta'$ depends only on $\beta$ and $z$ while $\tau\in T$ was arbitrary, the point $z$ is an equicontinuity point.
Thus every point in $X$ is an equicontinuity point, whence the flow $(X,T)$ is equicontinuous.

Conversely, suppose that the flow $(X,T)$ is equicontinuous, let $\alpha\in\mathcal U$ be an arbitrary entourage and choose $\alpha'\in\mathcal U$ with $\alpha'^{-1}\alpha'\subset\alpha$.
Let $\beta\in\mathcal U$ with $\beta\subset\alpha'$ be an entourage so that $(x,y)\in\beta$ implies $(tx,ty)\in\alpha'$ for all $t\in T$.
Then $(t x,x)\in\beta$ and $(y,x)\in\beta$ for some $t\in T$ imply $(t y,x)\in\alpha'^{-1}\beta\subset\alpha'^{-1}\alpha'\subset\alpha$, whence the Furstenberg condition holds.

Let $(X,T)$ be a equicontinuous flow, let the entourage $\alpha\in\mathcal U$ be arbitrary, and let $\beta\in\mathcal U$ be an entourage such that $(x,y)\in\beta$ is sufficient for $(t x,t y)\in\alpha$ for all $t\in T$.
Given $x,y\in X$ and a net $\{t_i\}_{i\in I}\subset T$ with $t x_i\to y$ we can conclude that
\begin{equation*}
(t_i^{-1} y,t_i^{-1}t_i x)=(t_i^{-1} y,x)\in\alpha
\end{equation*}
for all elements $i\in I$ with $(t_i x, y)\in\beta$.
Since the entourage $\alpha$ was arbitrary and the $\alpha$-neighborhoods of $x$ for all $\alpha\in\mathcal U$ define a neighborhood base at $x$, we have $t_i^{-1} y\to x$.
Therefore every point in $X$ is an almost automorphic point.
\end{proof}

Apart from the paper \cite{Ve1}, Veech studied minimal flows with almost automorphic points and an Abelian group $T$ in terms of a relation on $X$, which will be a subject of this note later.
From Theorem 1.2 in \cite{Ve2} it follows that a minimal flow $(X,T)$ where every point in $X$ is almost automorphic is an equicontinuous flow.
However, we want to present a streamlined proof even for non-Abelian $T$ and flows not necessarily minimal, in terms of the algebraic theory of topological dynamical systems.

\begin{proposition}\label{prop:inv}
Let $(X,T)$ be a flow so that \emph{every} point in $X$ is almost automorphic.
Then the enveloping semigroup $\mathcal E(X,T)$ is a group, and the operation of group inversion is a continuous mapping from $\mathcal E(X,T)$ onto $\mathcal E(X,T)$ with respect to the product topology on $X^X$.
\end{proposition}

\begin{lemma}
Suppose that every point in $X$ is almost automorphic.
Then $(X,T)$ is distal whence the enveloping semigroup $\mathcal E(X,T)$ is a group.
If $\{t_i\}_{i\in I}\subset T$ is a net with $t_i\to g\in\mathcal E(X,T)$, then the net $\{t_i^{-1}\}_{i\in I}\subset T$ is convergent to $g^{-1}\in\mathcal E(X,T)$.
\end{lemma}

\begin{proof}
Since every point in $X$ is almost automorphic and thus a distal point, the flow $(X,T)$ is clearly distal and the enveloping semigroup $\mathcal E(X,T)$ is a group.
Moreover, for every $x\in X$ holds $t_i^{-1}(g x)\to x$.
Since $\mathcal E(X,T)$ is a group, for every $y\in X$ there is a unique $x\in X$ with $gx=y$.
Thus we have $t_i^{-1} y=t_i^{-1}(gx)\to x=g^{-1} y$ for every $y\in X$, and the net $\{t_i^{-1}\}_{i\in I}\subset T$ defines the element $g^{-1}\in\mathcal E(X,T)$ as its limit point in $X^X$.
\end{proof}

\begin{proof}[Proof of Proposition \ref{prop:inv}]
Suppose that $\{g_j\}_{j\in J}\subset\mathcal E(X,T)$ is a net with $g_j\to g\in\mathcal E(X,T)$, and $\{g_i\}_{i\in I}$ is a subnet with $g_i^{-1}\to g'\in\mathcal E(X,T)$.
Let $W\subset\mathcal E(X,T)$ be an arbitrary open neighborhood of $g'$.
We may suppose that $g_i^{-1}\in W$ for every $i\in I$ and choose for every $i\in I$ an element $t_i\in T$ sufficiently close to $g_i$, so that $t_i\to g$ and, by the lemma, $t_i^{-1}\in W$.
The lemma implies also $t_i^{-1}\to g^{-1}$ so $g^{-1}\in\overline W$, and since $\mathcal E(X,T)$ is a Hausdorff space $g'$ and $g^{-1}$ have to be identical.
\end{proof}

\begin{remark*}
Though in general $\mathcal E(X,T)$ acts just algebraically and not continuously on $X$, the condition that every point is almost automorphic can be defined with respect to the action of $\mathcal E(X,T)$ on $X$ as well.
From the arguments above and the inclusion of $T$ in $\mathcal E(X,T)$ it follows then immediately that both conditions are equivalent.
\end{remark*}

\begin{theorem}\label{th:cp}
For a compact Hausdorff flow $(X,T)$ every point in $X$ is almost automorphic if and only if $\mathcal E(X,T)$ is a compact Hausdorff topological group in the topology of pointwise convergence in $X^X$.
\end{theorem}

\begin{proof}
Suppose that every point in $X$ is almost automorphic.
By Proposition \ref{prop:inv} the set $\mathcal E(X,T)$ is a group with left-continuous multiplication and continuous group inversion $g\mapsto g^{-1}$.
The algebraic identity $gh=(h^{-1} g^{-1})^{-1}$ implies that the group multiplication in $\mathcal E(X,T)$ is separately continuous.
By Ellis' joint continuity theorem \cite{El1} the group $\mathcal E(X,T)$ is a compact topological group.

Suppose that $\mathcal E(X,T)$ is a compact Hausdorff topological group, and let $x\in X$ with a net $\{t_j\}_{j\in J}\subset T$ so that $t_j x\to y\in X$.
By the compactness of $X^X$ there exists a subnet $\{t_i\}_{i\in I}$ with $t_i\to g\in\mathcal E(X,T)$, and the continuity of group inversion in $\mathcal E(X,T)$ implies that $t_i^{-1}\to g^{-1}$, hence $t_i^{-1}y\to g^{-1}y=x$.
Thus $t_j^{-1}y\to x$ holds also for the net $\{t_j\}_{j\in J}\subset T$, since otherwise (for a subnet) $t_j^{-1}y\to x'\neq x$.
\end{proof}

It is has been shown for a minimal flow $(X,T)$ that the continuity of all elements of $\mathcal E(X,T)$ as mappings on $X$ is equivalent to the equicontinuity of the flow.
It remains to prove that the elements of $\mathcal E(X,T)$ act continuously and not just algebraically on the minimal orbit closures in $X$.

\begin{theorem}\label{th:eaa}
A compact Hausdorff flow $(X,T)$ has the property that \emph{every} point is almost automorphic if and only if it is distal and every minimal set is equicontinuous.
\end{theorem}

\begin{proof}
Suppose that every point in $X$ is almost automorphic.
By Theorem \ref{th:cp} $\mathcal E(X,T)$ is a compact topological group, however acting algebraically on $(X,T)$.
For every point $x\in X$ the isotropy subgroup
\begin{equation*}
\mathcal E_x=\{g\in\mathcal E(X,T):gx=x\}
\end{equation*}
is closed in the product topology on $X^X$, and for every $g\in\mathcal E(X,T)$ it holds that $g\mathcal E_x g^{-1}\subset\mathcal E_{gx}$.
For every $x\in X$ we have $\{gx:g\in\mathcal E(X,T)\}=\bar{\mathcal O}_T(x)$, and thus the isotropy subgroups are conjugate on every orbit closure in $X$.
Hence $\mathcal E_{y}=g\mathcal E_x g^{-1}$ holds for $y=gx\in\bar{\mathcal O}_T(x)$ with suitable $g\in\mathcal E(X,T)$.
For a fixed $x\in X$ the map $g\mathcal E_x\mapsto gx$ is well defined, one to one, and continuous from the compact homogeneous space $\mathcal E(X,T)/\mathcal E_x$ onto $\bar{\mathcal O}_T(x)$.
This mapping is a homeomorphism making the flows $(\mathcal E(X,T)/\mathcal E_x,T)$ and $(\bar{\mathcal O}_T(x),T)$ isomorphic.
Since the enveloping semigroup of the flow $(\mathcal E(X,T),T)$ can be represented by left multiplication, the enveloping semigroup of $(\mathcal E(X,T)/\mathcal E_x,T)$ is given by the left multiplication of $\mathcal E(X,T)$ on the homogeneous space $\mathcal E(X,T)/\mathcal E_x$.
Therefore $\mathcal E(X,T)$ acts continuously on the homogeneous space $\mathcal E(X,T)/\mathcal E_x$, and by isomorphy its action on $\bar{\mathcal O}_T(x)$ is also continuous.
Since the enveloping semigroup consists of mappings which are continuous restricted on every orbit closure in $(X,T)$, every orbit closure in $(X,T)$ is equicontinuous.

Suppose that the flow $(X,T)$ is distal and every minimal set is equicontinuous, and let $x\in X$ and $\{t_i\}_{i\in I}\subset T$ with $t_i x\to y$ be arbitrary.
Then $x$ and $y$ are clearly within the same minimal orbit closure whence equicontinuity applies, and by Theorem \ref{th:fc} follows that $x$ is an almost automorphic point.
\end{proof}

\begin{cor}\label{cor:epaa}
If $(X,T)$ is transitive and every point is almost automorphic, then it is minimal and equicontinuous.
\end{cor}

With the following example we want to point out that the above statement cannot be strengthened for non-transitive flows, i.e. equicontinuity might hold restricted onto minimal sets but not globally:

\begin{example}\label{ex:rot}
Consider the unit disk with the rotation by the angle $2\pi r$ for a point $r \exp(i\varphi)$ with $r\in[0,1], \varphi\in[0,2\pi)$, i.e. $r \exp(i\varphi)\mapsto r \exp(i\varphi+i 2\pi r)$.
This homeomorphism is an irrational rotation on the circle with radius $r$ for irrational $r$, while it is a periodic rotation for rational $r$.
Moreover, for a positive integer $n$ the $n$-th iteration of this homeomorphism maps $r \exp(i\varphi)$ to $r \exp(i\varphi+i 2n\pi r)$, whence the $\Z$-action is not equicontinuous.
However, every point is almost automorphic since the homeomorphism acts even isometrically by rotation on each minimal orbit closure.
\end{example}

We want to proceed with the study of equicontinuity in terms of different relations on $X$ arising from dynamical properties of the flow $(X,T)$.
First we recall the definition of the \emph{regionally proximal relation} $RP$ and the \emph{proximal relation} $P$.

\begin{defi}
If $(X,T)$ is a flow then $(x,y)\in RP$ if there are nets $\{x_i\}_{i\in I}\subset X$, $\{y_i\}_{i\in I}\subset X$, $\{t_i\}_{i\in I}\subset T$, and a point $z\in X$ with $x_i\to x$, $y_i\to y$, $t_i x_i\to z$, and $t_i y_i\to z$.
By putting $x_i=x$ and $y_i=y$ above we obtain the proximal relation $P$.
\end{defi}

The relation $RP$ is reflexive, symmetric, $T$-invariant, and closed, but is not in general transitive.
It is easily checked that a flow $(X,T)$ is equicontinuous if and only if $RP=\Delta$, the diagonal (that is, $(X,T)$ has no non-trivial regionally proximal pairs).
If $(X,T)$ is a flow, then in an abstract way there exists always a smallest closed $T$-invariant equivalence relation $S_{eq}$ such that the quotient flow $(X/S_{eq},T)$ is equicontinuous.
This relation is called the \emph{equicontinuous structure relation}.
Clearly $RP\subset S_{eq}$, since there cannot be any non-trivial regionally proximal pairs in $X/S_{eq}$.
Using Ellis' joint continuity theorem \cite{El1} it can be shown that $S_{eq}$ is the smallest closed $T$-invariant equivalence relation which contains $RP$.
In particular, if the relation $RP$ is transitive and thus already an equivalence relation, then $S_{eq}=RP$.

Now we turn our attention to minimal flows, where it is frequently the case that in fact $RP$ is an equivalence relation.
There have been several proofs of this, under varying hypotheses.
The most relevant for our considerations is due to McMahon \cite{McM}, who proved that $RP$ is an equivalence relation in a minimal flow with an \emph{invariant probability measure} (cf. also \cite{Au}, p. 130, Theorem 8).
In particular this is the case whenever the group $T$ is Abelian, or even amenable.
Furthermore, an invariant probability measure exists for every distal minimal compact metric flow by Furstenberg's structure theorem \cite{Fu}, without any assumption on the acting group $T$.
Furstenberg's proof extends the point mass on the trivial flow successively via the Haar measures on the compact homogeneous spaces of the isometric extensions in the structure theorem, and transfinite induction with convergence in the weak topology give rise to an invariant probability measure.
This proof can also be generalized for distal minimal compact Hausdorff flows which are not metric.

In addition, McMahon's proof shows that one of the nets in the definition of regional proximality can be chosen to be constant.

\begin{fact}[\cite{McM}, cf. also \cite{Au}, p. 131]\label{f:McM}
Suppose that the compact minimal flow $(X,T)$ has an invariant probability measure.
If $(x,y)\in RP$, then there is a net $\{y_i\}_{i\in I}\subset X$ with $y_i\to y$ and a net $\{t_i\}_{i\in I}\subset T$ with $t_i x\to z$ and $t_i y_i \to z$ for some $z\in X$.
\end{fact}

Another characterization of the equicontinuous structure relation is due to Veech \cite{Ve1}, \cite{Ve2}:

\begin{defi}
We say that $(x,y)$ is in the \emph{``Veech''} relation $V$ if there is a $z\in X$ and a net $\{t_i\}_{i\in I}\subset T$ such that $t_i x\to z$ and $t_i^{-1} z\to y$.
Note that $V\subset RP$, since we have $(x,t_i^{-1}z)\to (x, y)$ and $t_i (x, t_i^{-1}z)\to (z, z)$.
Obviously $V(x)=\{x\}$ if and only if $x$ is an almost automorphic point.
\end{defi}

\begin{remarks*}
If $(x,y)\in V$, an element $\tau\in T$, and a net $\{t_i\}_{i\in I}\subset T$ with $t_i x\to z$ and $t_i^{-1} z\to y$ are given, then the net $\{t'_i=t_i\tau^{-1}\}_{i\in I}$ fulfills that $t'_i(\tau x)\to z$ and ${t'_i}^{-1} z=\tau\tau_i^{-1} z\to\tau y$.
Hence $(\tau x,\tau y)\in V$ and the relation $V$ is $T$-invariant.

At first view and in the general case, the properties of the relation $V$, except being reflexive and $T$-invariant, are unclear.

However, the relation $V$ enjoys an important property with respect to homomorphisms.
If $\pi:X\longrightarrow Y$ is a homomorphism of flows, then easily follows from the compactness of fibers of $\pi$ that $\pi(V_X)=V_Y$.
Thus if $(X,T)$ is a minimal compact Hausdorff flow and $S$ is the closed invariant equivalence relation generated by $V_X$, then $Y=X/S$ fulfills $V_Y=\Delta$.
By Corollary \ref{cor:epaa} the minimal flow $(Y,T)$ is equicontinuous, so $S$ is the equicontinuous structure relation of $(X,T)$.
\end{remarks*}

What Veech proved is that if $(X,T)$ is minimal, with $X$ metric and $T$ Abelian, then $V=S_{eq}$, the equicontinuous structure relation (so in fact $V=RP$ in this case).
To underscore the remarkable nature of this result, note that it is not clear from the definition that $V$ is even symmetric.
Veech's proof is lengthy and complicated, in the spirit of harmonic analysis.
Nowhere in his paper does he mention regional proximality.
Using McMahon's characterization of $RP$ in Fact \ref{f:McM}, we show the following result with the requirement that the phase space is metric:

\begin{theorem}\label{th:V=S}
Suppose that $(X,T)$ is a minimal flow with a compact Hausdorff phase space $X$ and an invariant probability measure, then $V(x)$ is dense in $RP(x)$ for every $x\in X$.

Moreover, for a metric phase space $X$ the relations $V$ and $RP$ coincide, and then by McMahon's result that $RP=S_{eq}$ follows also $V=S_{eq}$.
\end{theorem}

\begin{proof}
The inclusion $V\subset RP$ holds obviously, therefore it remains to prove the density of $V(x)$ in $RP(x)$ for every $x\in X$.
Note that if $(x,y)\in RP$, $U$ a neighborhood of $y$ and $W$ a non-empty open set in $X$, there is a $y'\in U$ and $t\in T$ such that $tx$ and $ty'$ are in $W$.
Indeed, let $\{y_j\}\subset X$ and $\{t_j\}\subset T$ be nets with $t_j(x,y_j)\to (z,z)$ for some $z\in X$, according to Fact \ref{f:McM}.
Let $\tau\in T$ such that $\tau z\in W$.
Then $\tau t_j(x, y_j)\to (\tau z, \tau z)$.
For a suitable $j$ we have, with $t=\tau t_j$ and $y'=y_j\in U$, $tx$ and $ty'$ in $W$.

Now let $(x,y)\in RP$ and fix a pseudometric $d_i$ out of the family $\mathcal D$ generating the uniform structure $\mathcal U$ on $X$, and define a decreasing sequence of neighborhoods $\{U_j\}$ of $y$ by $U_j=\{z\in X:d_i(z,y)<2^{-j}\}$.
We will define a decreasing sequence of open sets $\{W_j\}$.
Let $W_1$ be a non-empty open set.
By the sub-lemma above there is a $y_1\in U_1$ and a $t_1\in T$ such that $t_1 x$ and $t_1 y_1$ are in $W_1$.
Therefore $W_1\cap t_1 U_1\neq\emptyset$.
Let $W_2$ be non-empty open with $\overline{W}_2\subset W_1\cap t_1 U_1$.
By the sub-lemma there are $y_2\in U_2$ and $t_2\in T$ with $t_2 x$ and $t_2 y_2$ in $W_2$.
Then $W_2\cap t_2 U_2\neq\emptyset$.
Inductively, choose $W_{j+1}$ non-empty open $\overline{W}_{j+1}\subset W_j \cap t_j U_j$, and then let $y_{j+1} \in U_{j+1}$ and $t_{j+1}$ in $T$ with $t_{j+1} x$ and $t_{j+1} y_{j+1}$ in $W_{j+1}$.
By compactness of $X$ there exists a subsequence $\{t_{j_k}\}_{k\geq 1}$ of $\{t_j\}_{j\geq 1}$ so that $t_{j_k}x$ is convergent to some point $z\in X$.
Since $\{\overline{W}_j\}$ is a decreasing sequence of non-empty closed subsets we have $z\in\cap_{j\geq 1}\overline{W}_j$.
Therefore $z\in t_{j_k} U_{j_k}$ for all $k\geq 1$, hence $t_{j_k}^{-1}z\in U_{j_k}$ so $d_i(t_{j_k}^{-1}z,y)\to 0$.
By the compactness of $X$ there exists a subsequence of $t_{j_k}^{-1}z$ convergent to a point $y'$ in $V(x)$, and by the continuity of the mapping $z\mapsto d_i(z,y)$ follows $d_i(y',y)=0$.

We have verified that for $y\in RP(x)$ and every pseudometric $d_i\in\mathcal D$ there exists an element $y'\in V(x)$ with $d_i(y',y)=0$.
Given an arbitrary neighborhood $U$ of $y$ we select $d_i\in\mathcal D$ and $\varepsilon>0$ with $\{z\in X:d_i(z,y)<\varepsilon\}\subset U$, and therefore $V(x)\cap U\neq\emptyset$.
Since $y\in RP(x)$ and $U$ were arbitrary, the density of $V(x)$ in $RP(x)$ follows.

The statement regarding a metric phase space $(X,d)$ follows immediately from the above proof with $d_i=d$, since the metric $d$ separates points.
\end{proof}

\begin{remarks*}
The above proof also shows, without any assumption on the acting group $T$ and the compact Hausdorff space $X$, that $P(x)\subset\overline{V(x)}$ (if $(x,y)\in P$ just let $y_j=y$ be the constant net in the sub-lemma).
Moreover, if the phase space is metric, then $P\subset V$.

The above result is not valid without the assumption of minimality.
In Example \ref{ex:rot} we have $V(x)=\{x\}$ for every $x\in X$, while the set $RP(r \exp(i\varphi))$ is equal the circle with radius $r$ for all $r\in[0,1]$ and $\varphi\in [0,2\pi)$.

Let $\pi: X\longrightarrow Y$ be a homomorphism of the minimal flows $(X,T)$ and $(Y,T)$.
Using the so-called semi-open property of homomorphisms of minimal flows it can be verified that $\pi(RP_X)=RP_Y$.
Under the assumptions of the above theorem, this identity follows immediately from the identity $\pi(V_X)=V_Y$ and the closedness of the regionally proximal relation.

An example due to McMahon (cf. \cite{Au}, p.131) presents an action of the free group of two generators on the torus preserving no probability measure, and where the regionally proximal relation fails to be transitive.
Additionally, this flow is not proximal, while for a proximal flow the proof above shows $RP=P=V=X\times X$.
However, even in this example the statement according to Fact \ref{f:McM} is fulfilled for every regionally proximal pair, whence the proof above verifies that $V=RP$.
It is not known to us whether a counterexample to Theorem \ref{th:V=S} exists in the setting of a minimal flow.
\end{remarks*}

\begin{cor}
Suppose that $(X,T)$ is a minimal distal flow and suppose that $z\in X$ is an almost automorphic point.
Then every point in $X$ is almost automorphic and therefore the flow $(X,T)$ is equicontinuous.
Thus for a minimal (or just topologically transitive) flow $(X,T)$ the following are equivalent:
\begin{enumerate}
\item $(X,T)$ is equicontinuous
\item All points are almost automorphic
\item $(X,T)$ is distal and some point is almost automorphic
\end{enumerate}
\end{cor}

\begin{proof}
First we observe for a minimal flow $(X,T)$ with two points $x,y\in X$ distal but $(x,y)\in RP$, that every point $z\in X$ has a distinct point $z'$ with $(z,z')\in RP$.
Indeed, given a net $\{t_i\}_{i\in I}\subset T$ with $t_i x\to z$, we can assume by the distality of $x$ and $y$ that $t_i y\to z'\neq z$.
Since $RP$ is a closed and $T$-invariant relation, it follows that $(z,z')\in RP$.

Now we suppose that some $x\in X$ is not an almost automorphic point for a distal minimal flow $(X,T)$.
Since $V\subset RP$ there is a point $y\in X$ distal and regionally proximal to $x$.
By the argument above there is a point $z'\neq z$ regionally proximal to the almost automorphic point $z$.
Since $V(z)$ is dense in $RP(z)$ by Theorem \ref{th:V=S}, we can conclude that there exists a point $z''\in X$ distinct from $z$ so that $(z,z'')\in V$, which contradicts to the almost automorphy of $z$.
Thus every point in $X$ is almost automorphic and $(X,T)$ is minimal, and equicontinuity follows from Theorem \ref{th:eaa}.
\end{proof}

In addition to $V$ (which he calls $E$ in \cite{Ve2}) Veech defines a relation $D$ which he shows is the equicontinuous structure relation in the general (not necessarily metric) case.

\begin{defi}
We let $(x,y)\in D$ if there are nets $\{a_i\}_{i\in I}, \{b_i\}_{i\in I}\subset T$ with $a_i x\to x$, $b_i x\to x$, $b^{-1} a_i x\to y$.
\end{defi}

We give an alternate proof of Veech's result, using McMahon's characterization of the regionally proximal relation in Fact \ref{f:McM}.

\begin{theorem}
We have the following inclusions:
\begin{enumerate}
\item $D \subset RP$.
\item If $(X,T)$ is minimal, then $V \subset D$.
\item If the minimal compact Hausdorff flow $(X,T)$ has an invariant probability measure, then $D=RP$ (so $D$ is the equicontinuous structure relation).
\item If $(X,T)$ is a minimal compact \emph{metric} flow with an invariant probability measure, then $V=RP=D$.
\end{enumerate}
\end{theorem}

\begin{proof}
(i):
Let $(x,y)\in D$, so there are nets $\{a_i\}_{i\in I}, \{b_i\}_{i\in I}\subset T$ with $a_i x\to x$, $b_i x\to x$, $b_i^{-1} a_i x\to y$.
Then we have $(x,b_i^{-1}a_i x)\to(x,y)$ and $b_i (x,b_i^{-1} a_i x)\to (x,x)$.
Thus $(x,y)\in RP$.

(ii):
Let $(x,y)\in V$, let the entourage $\alpha\in\mathcal U$ be arbitrary, and let $\alpha'\in\mathcal U$ be an entourage with $\alpha'^2\subset\alpha$.
By the minimality of $(X,T)$ there is a $z\in X$ with $(x,z)\in\alpha'$ and a $t\in T$ with $(t x,z)\in\alpha'$ and $(t^{-1}z,y)\in\alpha'$.
Let the entourage $\beta\in\mathcal U$ (with $\beta\subset\alpha'$) correspond to $\alpha'$ for $t^{-1}$, and let $s\in T$ such that $(sx,z)\in\beta$.
Then we have $(t^{-1} sx,t^{-1} z)\in\alpha'$.
So we have $(t x, x)\in\alpha'\alpha'^{-1}\subset\alpha$, $(sx,x)\in\beta\alpha'^{-1}\subset\alpha$, and $(t^{-1} sx,y)\in\alpha'^2\subset\alpha$.
Since the entourage $\alpha$ was arbitrary, it follows that $(x,y)\in D$.

(iii):
Let $(x,y)\in RP$.
Let $U$ and $W$ be neighborhoods of $x$ and $y$ respectively.
By Fact \ref{f:McM} there is a $y'\in W$ and a $t\in T$ such that $t x$ and $t y'$ are in $U$.
Let $s\in T$ such that $sx\in U$ and so close to $t y'$ so that $t^{-1} sx\in W$.
Summarizing, we have $tx, sx\in U$ close to $x$ and $t^{-1} sx\in W$ close to $y$.
Since the neighborhoods $U$ of $x$ and $W$ of $y$ were arbitrary, we have $(x,y)\in D$.

(iv):
This follows from Theorem \ref{th:V=S}.
\end{proof}

\noindent\textbf{Acknowledgments.}
The authors would like to thank Ethan Akin and Benjamin Weiss for useful conversations.

\end{document}